\documentclass{amsart}

\usepackage{amsmath,amsthm,amssymb,amsfonts, epsfig}
\usepackage{setspace, stmaryrd, verbatim, euro, enumerate}
\usepackage{bbm, tensor,comment,textcomp, todonotes}
\usepackage{hyperref}

\usepackage{color}

\newtheorem{theorem}{Theorem}[section]

\newtheorem{conjecture}[theorem]{Conjecture}
\newtheorem{corollary}[theorem]{Corollary}

\newtheorem{definition}[theorem]{Definition}
\newtheorem{lemma}[theorem]{Lemma}
\newtheorem{proposition}[theorem]{Proposition}

\theoremstyle{definition}
\newtheorem{example}[theorem]{Example}
\theoremstyle{example}
\newtheorem{remark}[theorem]{Remark}

\newcommand\cB{\mathcal{B}}
\newcommand\cC{\mathcal{C}}
\newcommand\cD{\mathcal{D}}

\newcommand\cF{\mathcal{F}}
\newcommand\cG{\mathcal{G}}
\newcommand\cH{\mathcal{H}}
\newcommand\cI{\mathcal{I}}

\newcommand\cM{\mathcal{M}}
\newcommand\cN{\mathcal{N}}
\newcommand\cO{\mathcal{O}}

\newcommand\cV{\mathcal{V}}

\newcommand\bB{\mathbb{R}^N}

\newcommand\bE{\mathbb{E}}

\newcommand\bN{\mathbb{N}}

\newcommand\bR{\mathbb{R}}

\newcommand\an[1]{\langle #1 \rangle}
 
\DeclareMathOperator{\co}{co} 
\DeclareMathOperator{\aff}{aff} 
\DeclareMathOperator{\ri}{ri} 
\DeclareMathOperator{\gr}{gr} 
\DeclareMathOperator{\domn}{\cD(\mu,\nu)}
\DeclareMathOperator{\domnp}{\cD(\mu',\nu')}
\newcommand{\cvcmp}{C_x(\mu,\nu)}
\newcommand{\cnvmap}{E}

\usepackage{accents}
\DeclareMathOperator{\Lip}{Lip}
\newcommand\Lz{\text{Lipschitz }}

\newcommand\MN{\cM}

\DeclareMathOperator{\cc}{CC(\bB)}

\title{Structure of martingale transports in finite dimensions}
\author{Jan Ob{\l}\'oj \and Pietro Siorpaes}
\date{\today{}}

\address{Ob\l\'oj: Mathematical Institute, University of Oxford, Oxford OX2 6GG}
\email{jan.obloj@maths.ox.ac.uk}
\urladdr{www.maths.ox.ac.uk/people/jan.obloj}
\address{Siorpaes: Department of Mathematics, Imperial College London, London SW7 2AZ}
\email{p.siorpaes@imperial.ac.uk}
\urladdr{www.imperial.ac.uk/people/pietro.siorpaes}

\begin{document}

\begin{abstract}
We study the structure of martingale transports in finite dimensions. We consider the family $\cM(\mu,\nu) $ of  martingale measures on $\bB\times \bB$ with given marginals $\mu,\nu$, and construct a family of relatively open convex sets $\{C_x:x\in \bB\}$, which forms a partition of $\bB$, and such that any martingale transport in $\cM(\mu,\nu) $ sends mass from $x$ to within $\overline{C_x}$, $\mu(dx)$--a.e. 
Our results extend the analogous one-dimensional results of Beiglb\"ock and Juillet \cite{BJ:16} and Beiglb\"ock et al. \cite{BeNuTu15}. We conjecture that the decomposition is canonical and minimal in the sense that it allows to  characterise the martingale polar sets, i.e.\ the sets which have zero mass under all measures in $\cM(\mu,\nu)$,  and offers the martingale analogue of the characterisation of transport polar sets proved in \cite{BeGoMaSc08}.\medskip\\
{\sc Note.} \emph{This work is made publicly available simultaneously to, and in mutual recognition of, a parallel and independent work \cite{DeMT:17} which studies the same questions. In due course, we plan to release an amended version proving the conjectured minimality of our convex partition.}
\end{abstract}
\thanks{This research has been generously supported by the European Research Council under the European Union's Seventh Framework Programme (FP7/2007-2013) / ERC grant agreement no.\ 335421. Jan Ob\l\'oj is also grateful to St John's College in Oxford for their financial support.}
\maketitle

\section{Introduction}

Optimal transportation is a classical and influential field in mathematics. Its origins trace back to Gaspard Monge, while its modern incarnation was born from works of Kantorovich. Since then, it has seen tremendous advances, in particular in understanding the geometry of the optimal transport maps, through the works of Brenier, Gangbo, McCann, Otto, Villani and many others, see Villani \cite{VillaniOldNew} for an extensive account of the theory. More recently, inspired by a number of applications in probability theory as well as financial mathematics, new variants of the problem have been studied under various constraints on the transports maps. 
The most notable example is  the so-called \emph{martingale optimal transport} (MOT) problem, where the transport dynamics have to obey the martingale condition. Martingales are used in mathematical finance to model dynamics of price processes and marginal specification can be seen as equivalent to knowing sufficiently many market prices of simple (European) options. If the cost functional being optimised is given by the payoff of another (exotic) option then MOT problem values correspond to the robust (no-arbitrage) bounds for such an exotic option. Such rephrasing of robust pricing as an MOT problem was achieved by Beiglb\"{o}ck et al.\ \cite{BHLP:13} in discrete time, and Galichon et al.\ \cite{GalichonHenryLabordereTouzi:11} in continuous time. The latter worked in a  one dimensional setting where a martingale is a time-change of a Brownian motion. In this way, the MOT problem links naturally to the Skorokhod embedding problem, a well studied topic in probability theory, see Ob\l\'oj \cite{genealogia} for an account. Recently, in a beautiful display of how new developments can be achieved when an old theory is re-interpreted using entirely novel methods, Beiglb\"{o}ck et al.\ \cite{BCH:16} obtained a geometric description of supports of optimal Skorokhod embeddings, akin to the Gangbo and McCann \cite{GangboMcCann:96} characterisation for optimal transportation. 

Similarly to optimal transport, where the Kantorovich duality was a cornerstone result, it is of paramount interest to understand duality for the MOT problem. Partial results, under suitable continuity of the cost functional, were established in \cite{BHLP:13}, see also \cite{Za15Mo}. However the continuity assumption allowed one to side-step the problem of understanding and describing the polar sets, i.e.\ null sets under all martingale transport plans. For (all) transports such a description was given in Beiglb\"{o}ck et al.\ \cite{BeGoMaSc08}, as a corollary of the  complete description of Kantorovich duality provided by Kellerer \cite{Kellerer:84}. Recently, Beiglb\"{o}ck and Juillet \cite{BJ:16} and  Beiglb\"{o}ck et al.\ \cite{BeNuTu15} obtained analogous results for martingale transports in dimension one. Our aim here is to prove similar  results in a (arbitrary) finite dimension. As we shall see, the one dimensional picture is rather special and already in dimension two, the results are much richer and more involved. 

To highlight the features and state our main result, let us introduce some notation. 
A transport from $\mu$ to $\nu$ is a positive measure $\theta$ on $\bB \times \bB$ whose marginals are $\mu$ and $\nu$, and the set of all such transports is denoted by $\Pi(\mu,\nu)$. 
Clearly $\Pi(\mu,\nu)$ is non-empty as we may always take $\theta=\mu\otimes \nu\in \Pi(\mu,\nu)$. In \cite{BeGoMaSc08} it is  shown that the only polar sets of $\Pi(\mu,\nu)$ are the trivial ones:
\begin{equation}\label{eq:Kellerer}
 \theta(N)=0 \ \forall \theta\in \Pi(\mu,\nu)\quad \Longleftrightarrow\quad N\subseteq (\cN_\mu\times \bB) \cup (\bB\times \cN_\nu) ,
\end{equation}
for some  $\mu$--null   set $\cN_\mu$ and  $\nu$--null set $\cN_\nu$. Here, for two probability measures $\mu,\nu$ on $\bB$ with finite first moments, we consider the subset of martingale transports
$$ \cM(\mu,\nu) := \big\{\theta \in \Pi(\mu,\nu):\, \bE[Y|X]=X \textrm{ for }(X,Y)\sim \theta\, \big\}. $$
Equivalently, $\theta\in\cM(\mu,\nu)$ if and only if $\theta\in \Pi(\mu,\nu)$ and for any disintegration $\theta=\mu\otimes \gamma$, where $\gamma(x,\cdot)$ is a probability measure with a finite first moment, one has $\int y \gamma(x,dy)=x$ for $\mu$--a.e $x$. Note that $\cM(\mu,\nu)$ may often be empty. Jensen's inequality implies that if it is non-empty then $\mu$ and $\nu$ are in convex order, i.e. 
$$\int_{\bB} \phi d\nu \geq \int_{\bB} \phi d\mu  \text{ for every  convex  }\phi:\bR^N\to \bR ,$$
in which case we write $\mu\preceq_c \nu$; notice that  $\phi$ may fail to be integrable, but its negative part is integrable since a convex function is bounded below by an affine function\footnote{Indeed if $\phi$ is convex then its sub-differential at any point $x$ is non-empty, see \cite[Chapter 6]{HiLe93V1}.}, which is integrable under any measure with finite first moment. In a seminal work, Strassen \cite{Strassen:65} showed that the condition $\mu\preceq_c \nu$  is not only necessary but also sufficient to have $\cM(\mu,\nu)\neq \emptyset$. 

In dimension one a full description of MOT polar sets was given by \cite{BeNuTu15}, using the domain $\domn=\{x\in \bR: 0 <u_{\nu -\mu}(x)\}$, where $u_\lambda=|\cdot |\ast \lambda$ is the potential function associated to $\lambda$ (see \eqref{eq:potential} below), and $u_{\nu -\mu}=u_{\nu}- u_{\mu}\geq 0$ since $||\cdot -x ||$ is convex. We will now rephrase their results in the language used in  this paper. They showed that, under any $\theta=\mu\otimes \gamma \in\cM(\mu,\nu)$,  the mass from $x\in \domn$ may travel to the closure of the  connected component $C_x=\cvcmp$ of  $\domn$ containing $x$, i.e. $\gamma(x,\cdot)$ is concentrated on  $\overline{C_x}$; and the mass from $x\notin \domn$ is not moved, i.e. $\gamma(x,\cdot)$ is concentrated on $\overline{C_x}=C_x:=\{x\}$ (i.e. $\gamma(x,\cdot)=\delta_x$). Moreover, they showed that this is essentially\footnote{Notice the  difference between $C_x$ and $\overline{C_x}$ in our statements. In fact, in \cite{BeNuTu15} a slightly stronger statement is shown, as the authors can identify a set $J_x$ such that $C_x\subseteq J_x \subseteq  \overline{C_x}$ and the mass in $x$ has to stay in $J_x$ and can travel anywhere within $J_x$.}   the only limitation imposed by the martingale constraint, i.e. the mass from $x$ can travel \emph{anywhere} within $C_x$, meaning that the only martingale polar subsets of the graph $\gr(C):=\cup_{x\in \bR} \{x\}\times C_x$ of the multifunction $x \mapsto C_x$ are the trivial ones, similarly to \eqref{eq:Kellerer}. 

Notice that the potential functions are continuous, so the domain is open and thus the connected components are open disjoint intervals and they are at most countably many.  As it turns out, if one wishes to obtain a similar characterization of martingale polar sets in dimension greater than one, the nature of these components is necessarily more complex and intriguing.  
 Using simple examples, we argue that the components can be uncountably many and convexity, not connectedness, is their defining\footnote{Of course, a set $I\subseteq\bR$ is connected iff it is convex (and iff it is an interval); but in $\bR^N$ there are connected sets which are not convex.} property; also, different components may  have different Hausdorff dimensions. Having defined the domain  as $\cD:=\domn:=\{x\in \bB: \cvcmp\neq \{x\}\}$ (in analogy to what holds in the one dimensional case), we show that $\cvcmp$ can no longer be defined as the connected component of $x$ in $\domn$ (in fact it just cannot be defined using  $\domn$); indeed, we give an example of measures   $\mu\preceq_c \nu$ and  $\mu'\preceq_c \nu'$ with the same domain $\domn=\domnp$ but with drastically different components. 
In consequence, in arbitrary dimension the proper definition of the domain and its components is much more involved.

The first results in this direction were obtained by Ghoussoub et al. \cite{GKL:15}, who studied mainly the geometry of the optimal martingale transports in $\bB$ for the cost functions $\pm ||x-y||$. While their  focus was different from ours, they considered -- as we do -- a generally uncountable decomposition of the space into  relatively open convex sets, and  a corresponding disintegration of the measures involved. However,  their decomposition depends on $\theta\in \cM(\mu,\nu)$ and on a significant measurability assumption, and they work with a topology which is not separable which poses crucial difficulties since Polish structure is needed to consider disintegrations.
 Instead, we associate to each point $x$ a relatively open convex set $C_x \ni x$, which we call its (convex) component, which depends on $\nu -\mu$, but not on the choice of $\theta\in \cM(\mu,\nu)$. As we provide a constructive proof of existence of the convex components, we can give explicit examples in special cases of interest, and we plan to prove the measurability of the map $\overline{C}:x\mapsto \overline{C_x}$ in the future version of this paper. Moreover, having singled out the Wijsman topology as the appropriate one to consider, we do work with a Polish space, and we can rely on many results in the literature about measurability of multifunctions: indeed, $\overline{C}$ considered as a set-valued function is Borel measurable (with respect to the Wijsman topology on the family of closed convex sets) if and only if $\overline{C}$  considered as a multifunction is measurable.

Our main contribution is to provide the definition of convex components (see Definition \ref{def:convex_component} below), motivate it with examples,  and establish its use in describing the possible evolutions of martingale optimal transport plans by proving the following theorem.
\begin{theorem}\label{thm:main_result}
Let $\mu,\nu$ be two probability measures on $\bB$ in convex order  $\mu\preceq_c \nu$. 
Then the family $\{C_x:x\in \bB\}$ of convex components associated with $(\mu,\nu)$ forms a partition of $\bB$: $C_x=C_y$ or $C_x\cap C_y=\emptyset$, $x,y\in \bB$ and $x\in C_x$.
Each $C_x$ is convex and relatively open, and for every $\theta=\mu\otimes \gamma \in\cM(\mu,\nu)$, $\gamma(x,\cdot)$ is concentrated on  $\overline{C_x}$ for $\mu$ a.e. $x$, i.e. $\theta$ is concentrated on the graph  $\gr(\overline{C}):=\cup_{x\in \bR^N} \{x\}\times \overline{C_x}$.
\end{theorem}
We remark that it easily follows from Theorem \ref{thm:main_result} that  $\mu$ and $ \nu$ coincide on the complement of $\overline{\domn}$ (as it happens in dimension one), justifying further our definition of domain; see Corollary \ref{Mu=NuOutsideD}.
As a very special sub-case of our theorem above, we will obtain the following corollary, greatly generalising \cite[Theorem 7.A.15]{ShSh07Stochastic}, which amounts to the case where $\phi(y)=||y-b||^2$ (where $b$ is the common barycenter of $\mu$ and $\nu$, i.e. $b:=\int x \mu(dx)=\int x \nu(dx)$). 
\begin{corollary}
\label{MeasEqintro}
If  $\mu\preceq_c \nu$ and $\int \phi d\mu= \int \phi d\nu<\infty$ for a strictly convex $\phi:\bR^N \to \bR$ then $\nu=\mu$.
\end{corollary} 
We conjecture that the closures of the convex components, as we defined them, are the smallest possible closed convex sets on which all martingale transports are confined, meaning that the only martingale polar subsets of $\gr(C)$ are the trivial ones, i.e. those of the form \eqref{eq:Kellerer}. We state this as a conjecture and we are currently working towards completing its proof.
\begin{conjecture}\label{con:reverse}
 In the setting of Theorem \ref{thm:main_result}, $N\subseteq \gr(C)$ is $\cM(\mu,\nu)$--polar, i.e.\ $\theta(N)=0$ for all $\theta\in \cM(\mu,\nu)$, if and only if it is $\Pi(\mu,\nu)$--polar.
\end{conjecture}
Throughout the paper we work with several running examples which we use to illustrate the arising challenges and to motivate our definitions. These examples provide benchmark cases where we can identify what the convex components should be if they are to satisfy Theorem \ref{thm:main_result}, and hence they lead us towards developing a  general theory which  allows us to recover them as special cases. These examples are introduced in Section \ref{intro ex}. Subsequently, the paper is devoted to the definition of $C_x$ and its motivation and to the proof of the above theorem. We introduce notation and then, in Section \ref{sec:cnv_comp_results}, define  asymptotically affine components $A_x((\phi_n)_n)$ corresponding to some sequences convex functions $(\phi_n)_n$. Then, in Section \ref{sec:convex_components}, we construct the convex components  $C_x$ as  a certain essential intersection of such asymptotically affine components. To allow for a smooth narrative of the construction, many technical proofs are grouped in the subsequent Section \ref{sec:cnv_comp_proofs}. 
  
Finally we note that we have hoped to make this paper publicly available only after having proved Conjecture \ref{con:reverse}. However we have been recently made aware of a parallel and independent work of De March and Touzi \cite{DeMT:17} who study the same problem and obtain similar results using different techniques. 
Consequently, we have agreed to simultaneously make our works publicly available. 
  
  \section{Examples}
\label{intro ex}

\begin{example}\label{ex:discrete}
For $k\geq 2$, consider the following probability measures on $\bR^2$:
$$\mu^k:= \frac{1}{2k} \sum_{i=0}^{k-1}  2\delta_{x_i} , \quad \nu^k:= \frac{1}{2k} \sum_{i=0}^{k-1}\left(  \delta_{y_i^+} +  \delta_{y_i^-}\right),$$  
where $x_i:=(\frac{i}{k-1},0)$,  $y_i^{\pm}:=(\frac{i}{k-1},\pm 1)  \in [0,1]\times [-1,1]$. 
Define the kernel  $\gamma^\pm(x,\cdot)$ to be  $\frac{1}{2}(\delta_{(t,-1)}+\delta_{(t,1)})$ for $x=(t,0) \in  (0,1) \times \{0\}$ and to be $\delta_x$  otherwise. Note that, for every $x\in \bR^2$, $\gamma^\pm(x,\cdot)$ is a probability measure with finite first moments and with barycenter $x$  and that $\nu^k=\int \mu^k(dz) \gamma^\pm(z,\cdot)$. It follows that $\theta^k:= \mu^k \otimes \gamma^\pm\in \cM(\mu^k,\nu^k)$ and in particular $\mu^k\preceq_c\nu^k$, $2\leq k\leq \infty$. Further it is easy to see that $\cM(\mu^k,\nu^k)$ is a singleton and $\theta^k$ is the unique martingale transport connecting $\mu^k$ and $\nu^k$. Indeed, the martingale condition implies that the mass from $(0,0)$ may only go up and down -- it may not go right since $\nu^k$ puts no mass to the left, i.e.\ the atom in $(0,0)$ has to be distributed to the atoms in $\{0\}\times \{-1,1\}$. Iterating, we conclude. It follows Theorem \ref{thm:main_result} holds with $C_x=C_i:=\{\frac{i}{k-1} \} \times (-1,1)$ for $x\in C_i$, $i=0,\ldots, k-1$, and $C_x=\{x\}$ otherwise. Our general definitions have to reproduce this simple example.
\end{example}
\begin{example} \label{ex:continuous}
We consider the limiting case of Example \ref{ex:discrete} above. 
Let $\mu^\infty$ be uniform on $ (0,1)\times \{0\}$ and $\nu^\infty$ be uniform on $(0,1)\times \{-1,1\}$. In particular, 
$\mu^\infty=\lim_k \mu^k$ and $\nu^\infty=\lim_k \nu^k$. It is easy to see that $\theta^\infty=\mu^\infty\otimes \gamma^\pm$ is the unique element in $ \cM(\mu^\infty,\nu^\infty)$. It follows that Theorem \ref{thm:main_result} holds with 
$$C_{(t,s)}=\{t \} \times (-1,1),\quad \textrm{for all }(t,s)\in [0,1]\times (-1,1),$$
and $C_x=\{x\}$ otherwise. In particular there are uncountably many $(C_x)_{x\in \bR^2}$. Note also that, given a fixed Lebesgue null set $\Gamma$ in $(0,1)$, we could arbitrarily redefine $C_x$ for $x\in \Gamma$ and Theorem \ref{thm:main_result} would still hold. More generally, we observe that $x\to C_x$ is only determined $\mu(dx)$--a.e. We come back to this in Section \ref{sec:convex_components}. 
\begin{example}\label{ex:mixed}
Using notation of examples above, let 
 $\tilde\mu^k:=\frac{1}{2}(\mu^k+\delta_{(\frac{1}{2},0)})$, $2\leq k\leq \infty$.  
This case requires us to distinguish between even and odd $k$ since $\mu^k(\{(\frac{1}{2},0)\})>0$ iff $k$ is odd. For even $k$, or $k=\infty$, we let $\tilde \gamma^k(x,\cdot) = \gamma^\pm(x,\cdot)\mathbf{1}_{z\neq (\frac{1}{2},0)} + \nu^k(\cdot)\mathbf{1}_{x= (\frac{1}{2},0)}$. For odd $k$ we let 
$$\tilde \gamma^k(x,\cdot) = \gamma^\pm(x,\cdot)\mathbf{1}_{x\neq (\frac{1}{2},0)} + \left(\frac{k}{k+1}\nu^k(\cdot)+ \frac{1}{2(k+1)}(\delta_{(\frac{1}{2},-1)}+\delta_{(\frac{1}{2},1)})\right)\mathbf{1}_{x= (\frac{1}{2},0)}.$$
Observing that the barycentre of $\nu^k$ is $(\frac{1}{2},0)$, it follows instantly that each $\tilde \gamma^k(x,\cdot)\in \cM_1$ and has barycentre equal to $x$. In consequence, $\tilde\theta^k:=\tilde \mu^k\otimes \tilde \gamma^k\in \cM(\tilde \mu^k,\nu^k)$. In particular, the mass from the centre of the rectangle $[0,1]\times [-1,1]$ is spread to its corners. The convex components in Theorem \ref{thm:main_result} are the same for all $2\leq k\leq \infty$ and given by
\begin{itemize}
\item  
$C_{(0,s)}=\{0 \} \times (-1,1)$ and $C_{(1,s)}=\{1 \} \times (-1,1)$, for $s\in (-1,1)$,
\item
$C_{(t,s)}=(0,1) \times (-1,1)$ for $(t,s)\in (0,1)\times (-1,1),$ 
\item
$C_x=\{x\}$ for all other $x$.
\end{itemize} 
This example showcases two important features. First, the convex components may have different Hausdorff dimension for different $x\in \bB$. 
Second, the domains in this and in the previous example coincide: $\cD_{\mu^\infty,\nu^\infty}= \cD_{\tilde\mu^k,\nu^k}$, $2\leq k\leq \infty$, while the convex components are very different.  
\end{example}

\end{example}

\begin{example}
\label{ex:Gaussian}
If $\mu,\nu$ are Gaussian measures  on $\bR^n$ then $\mu\preceq_c \nu$ iff $\mu,\nu$ have the same mean and their covariance matrices  
$\Sigma_\mu, \Sigma_\nu$ are such that $\Sigma:=\Sigma_\nu -\Sigma_\mu$ is positive semidefinite (see \cite[Example 7.A.13]{ShSh07Stochastic}). By an orthogonal change of coordinates we can assume w.l.o.g. that $\Sigma$ is diagonal with eigenvalues $\sigma_1\geq \sigma_2 \geq \ldots \geq \sigma_n\geq 0$.
We will show that if $\Sigma$ is (strictly) positive definite then $\mu\preceq_c \nu$ are irreducible, i.e.\ $C_x=\bB$ for each $x\in \bB$. 
 More generally, let $k\in \{1,\ldots, n\}$ be such that $\sigma_i>0$ iff $i\leq k$; then the convex component $C_x$  of each point $x=(x_i)_i\in \bR^n$ is $\bR^k\times \{(x_{k+1}, \ldots, x_n)\}$.
\end{example}

\section{Convex components governing martingale transports}
\subsection{Notation}

We will denote with $\an{x,y}$ the usual dot product between $x,y\in \bB$, and with $||x||$ the associated Euclidian norm.
Given $W\subseteq \bB$, we will denote with $\co(W)$ its convex hull, with $\accentset{\circ}{W}$ its interior, with $\bar{W}$ its closure and with  $\partial W$ its border.
We will denote with  $\aff(V)$  (resp. $\ri(V)$) the affine hull  (resp. the relative interior)  of a convex set $V\subseteq \bB$, and with $I$ an arbitrary set of indices. We will denote with $B_{\epsilon}(x):= \{ y\in \bB: ||y-x||<\epsilon \}$ the open ball of radius $\epsilon>0$ centered in $x\in \bB$.
We will denote with $[x,y]$ (resp. $(x,y)$, $\overrightarrow{[x,y)}$, $\overleftrightarrow{(x,y)}$) the set $\{ x+t(y-x): t \in A\}$ with $A=[0,1]$ (resp. $(0,1) , [0,\infty), (-\infty, \infty)$).
Given  $K \subseteq \bB $,  $f:\bB \to \bR$ and $g:K \to \bR$, we denote with $f_{|K}$  the restriction of $f$ to $K$ and define\footnote{This definition makes sense if $K$ contains at least two points, which holds in the sequel any time we need to consider a \Lz constant.} $$\textstyle \Lip(g):=\sup_{x,y\in K, x\neq y} \frac{|g(x)-g(y)|}{||x-y||} $$ as the \Lz constant of  $g$ and set
\begin{align}
\label{defC}
\cC :=\{ \phi:\bB \to \bR  \text{ is convex and Lipschitz} \} , \quad \cC_+:=\{\phi \in \cC: \phi \geq 0  \}  .
\end{align} 
We will denote with $\cO$ the family of non-empty, convex and relatively open sets of $\bB$.
If $C\subseteq \bB$ is closed we define distance from $C$  as  $d_C(x):=\min_{y\in C} ||x-y||_{\bB} $, and we recall that if $C$ is convex then $d_C\in \cC_+$ (see \cite[Chapter 4, Example 1.3(c)]{HiLe93V1}) and, as   is easily seen, $\Lip(d_C)=1$.
If $\alpha$ is a real measure and $\phi$ a $\alpha$-integrable function, we often write  $\an{\alpha, \phi}$ for $\int \phi d\alpha $.
We denote with  $\MN(\bB)$ (or simply $\MN$) the set of  positive Borel measures $\alpha$ on $\bB$  which are finite  and have finite  first moment (i.e. are such that $\int_{\bB} (1+||x||)\alpha(dx)<\infty$), and with $\MN_1$ the set of probabilities in $\MN$.
Notice that if $f:\bB \to \bR$ is  \Lz and $\alpha\in \MN$ then $f\in L^1(\alpha)$.
Throughout the paper we consider a given pair $\mu,\nu\in \MN$ assumed to be in convex order
which we will write as  $\mu \preceq_c\nu$. 
We will use without further notice the fact that if $a$ is affine then $\an{\nu -\mu, \pm a}\geq 0$ and so $\an{\nu -\mu, a}=0$. Notice that the functionals $\an{\nu-\mu, \phi }$ and $\Lip( \phi_{|K})$, defined for $\phi \in \cC$, take values in $[0,\infty)$ and are positively homogenous. We recall that if $\phi:\bR \to \bR$ is convex,  
its right derivative $\phi'_+$  exists and is increasing and right continuous, and its
second derivative in the sense of Schwartz distributions is the positive Radon measure  $\phi''$  which satisfies $\phi''((c,d])=\phi'_+(d)-\phi'_+(c) $.

\subsection{Asymptotically affine components}
\label{sec:cnv_comp_results}
In dimension one all information needed to understand the structure of $\cM(\mu,\nu)$ is contained in the potentials function
\begin{equation}\label{eq:potential}
u_\lambda(x):= \int |x-y| \lambda(dy), \quad x\in \bR 
\end{equation}
of $\lambda:=\nu -\mu$.
The domain of $\mu \preceq_c\nu$, defined as the set $\{u_\nu >u_\mu \}=\{u_{\nu-\mu}>0 \} $, being open it is composed of at most countably many disjoint open intervals which are the convex components which delimit the martingale evolutions in $\cM(\mu,\nu)$.  
The key idea to generalise the study of martingale polar sets to dimension higher than one is that, instead of considering  $u_{\nu-\mu}(x)=\an{\nu-\mu, \phi_x}$ with $\phi_x=|\cdot -x|$, one should consider the wider family  $\an{\nu-\mu, \phi}$ where $\phi \in \cC$; restricting to the $\phi \in \cC_+$ such that $\phi(x)=x$ gives in a way an multidimensional equivalent of considering   $|\cdot -x|$. 
With this in mind, the following remark provides the crucial property which characterizes the convex components in a way that does not make any reference to the potential functions. 
\begin{remark}
\label{EqInDom}
In dimension one (i.e. if $\bB=\bR$) the following are equivalent
\begin{enumerate}
 \item  
 $\{ u_\nu > u_\mu \} \supseteq (c,d)$,
 \item 
if $\an{\nu-\mu, \phi} = 0$ for $\phi\in \cC$ then necessarily $\phi_{|(c,d)}$ is affine.
\item If $\an{\nu-\mu, \phi^n} \to 0$ for $(\phi^n)_{n \in \bN} \subseteq \cC$ then
there exist affine functions $(a_n)_n$ s.t. $ (\phi^n - a_n)(x) \to 0$  for all $x\in (c,d)$.
\end{enumerate} 
\end{remark} 
The above remark is easily proven using the identities (where $\phi\in \cC$)
$$\an{\nu-\mu, \phi} =\int_{\bR} (u_{\nu}-u_{\mu}) d \phi'' \text{ and } \phi''((c,d])=\phi'_+(d)-\phi'_+(c)  $$
and taking $a_n$ to be an affine function supporting $\phi$ at $x$.
 To obtain higher dimensional analogues of the concept of convex component, it is natural to start with the second property in Remark \ref{EqInDom}; the idea being essentially that the convex component of  $x$ should be the largest convex set on which all convex functions such that $\an{\nu-\mu, \phi} =0$ are affine. To make this more precise and prove the existence of such set, notice that if $\phi:\bR \to \bR$ is convex then there exists disjoint intervals $(a_n,b_n)$ such that $\phi$ is affine on each $[a_n,b_n]$ and is locally strictly convex on $\bR \setminus [a_n,b_n]$: indeed $[a_n,b_n)$ are the intervals of constancy of the increasing right continuous function $\phi'_+$, or equivalently $\bR \setminus (a_n,b_n)$ is the support of the measure $\phi''$. Thus, for any $x \in \bR \setminus [a_n,b_n]$ there exists no open interval containing $x$ on which $\phi$ is affine, whereas for $x\in (a_n,b_n)$ there exists the biggest open interval containing $x$ and on which $\phi$ is affine (it is indeed $(a_n,b_n)$). 
Notice that it is crucial that we insist on the intervals being open: if $\phi(x)=|x|$ there does not exists a biggest interval containing $0$ on which $\phi$ is affine, since $\phi$ is affine on both $(-\infty,0]$ and $[0,\infty)$  but not on their union.

So, one might conjecture that for any $\phi\in \cC$ and $x\in \bB$ there exists a largest set $A(\phi)_x\in \cO$ which contains $x$ and the family of such sets forms a partition of $\bB$. Building on the characterisation in Remark \ref{EqInDom} we would then expect that $C_x\subseteq A(\phi)_x$ for any $\phi$ such that $\an{\nu-\mu, \phi}=0$, and that  
$C_x$  should be defined as the intersection of $A(\phi)_x$ over such $\phi$. As we will see, this is `essentially' correct, but figuring out what  exactly is the correct construction turns out to be quite delicate.
First, for technical reasons related to the proof of Conjecture \ref{con:reverse}, we will work not with single functions $\phi$ such that  $\an{\nu-\mu, \phi}=0$, but rather  with sequences $(\phi^n)_n$ such that 
$\an{\nu-\mu, \phi^n}\to 0$. Second, it turns out that one should not take the intersection over all such sequences $(\phi^n)_n$, but rather an $\mu$-essential intersection, in a sense which we will motivate and explain later.
\begin{definition}
\label{def:as_aff}
We say that $(\phi^n)_{n\in \bN} \subseteq \cC$ is \emph{asymptotically affine on $V$} if there exist affine functions $(a_n)_n$ such that  $ \phi^n - a_n \to 0$ on $V$.
\end{definition} 
The intuitive meaning is that  if $\an{\nu-\mu, \phi^n} \to 0$ for $\phi^n\in \cC$ then $\phi^n_{|V}$ is `affine in the limit', meaning not  that $\phi^n$ converge to an affine function\footnote{Indeed the value of  $\an{\nu-\mu, \phi^n}$ does not change if we subtract from $\phi^n$ an arbitrary affine function $a^n$, and the $(a^n)_n$'s do not need to converge. Said otherwise, any sequence of affine functions is asymptotically affine, even if it is not converging.}, but rather that the $\phi^n$ are more and more `flat'. The above notion will ultimately allow us to define the multi-dimensional equivalents of the intervals $(c,d)$ of Remark \ref{EqInDom}.
\begin{proposition}\label{prop:def_aff_comp}
Fix $x\in \bB$ and $(\phi^n)_{n\in \bN} \subseteq \cC$. Then there exists the biggest, with respect to set inclusion, set in $\cO$ which contains $x$ and on which $(\phi^n)_n$ is asymptotically affine. Further, it is given by
\begin{align}\label{defA_x}
A_x((\phi^n)_n):=\ri(\co(\cup \{V\in \cO: x\in V  \text{ and $(\phi^n_{|V})_n$ is asymptotically affine} \} ))
\end{align} 
\end{proposition}
Observe that $(\phi^n)_n$ is asymptotically affine on any singleton, and if it is asymptotically affine on a set then it is asymptotically affine on any of its subsets. $A_x((\phi^n)_n)$ is called the \emph{$(\phi^n)_n$-asymptotically affine component}  of $x$ or, simply, the \emph{$(\phi^n)_n$-component} of $x$. When $(\phi^n)_n$ are fixed we write $A_x$. Indexed by $x\in \bB$, the family forms a partition of $\bB$ in the following sense.
\begin{definition}
For a set $\Gamma\subset \bB$, a family of sets $U_i$, $i\in \cI$ is said to be a \emph{convex partition of $\Gamma$} if for all $i\in\cI$, $U_i$ is convex, relatively open, 
$$\bigcup_{i\in \cI} U_i = \Gamma\quad \textrm{ and }\quad U_i\cap U_k\neq \emptyset \Longrightarrow U_i=U_j,\textrm{ for } i,j\in \cI.$$
\end{definition}
Define
$$\cV=\cV((\phi^n)_n):=\bB \setminus \{x\in \bB:  A_x((\phi^n)_n)=\{x\}\} ,$$ and notice that  $\cV$ equals the set of $x\in \bB$ for which there exist $y,z\in \bB, y\neq z$ such that $ x\in (y,z) \text{ and } (\phi^n)_n  \text{ is  asymptotically affine on } (y,z) $. It follows readily 
from Proposition \ref{prop:def_aff_comp} and its proof that 
\begin{proposition}\label{prop:PropOfA_x}
Let $(\phi^n)_{n\in \bN} \subseteq \cC$. 
The family of sets $\{A_x((\phi^n)_n)\}$ for $x\in \cV((\phi^n)_n)$ (resp.\ for $x\in \bB$) forms a convex partition of $\cV((\phi^n)_n)$ (resp.\ $\bB$). 
\end{proposition}

\begin{example}
\label{ }
We give an example of $(\phi^n)_n$-components which shows that, while  in dimension $1$ the $(\phi^n)_n$-components form  a countable partition of $\cV$ made of open intervals, even in dimension two and with constant sequence  $\phi^n=\phi$  the non-trivial $(\phi^n)_n$-components can have less than full dimension and can be uncountable. 
 Let $\Gamma$ be the half disk $\Gamma:=\{ (x,y)\in \bR^2: x^2+y^2\leq 1, y\geq 0\}$ and define $\phi:=d_\Gamma$. Then the family of all $\phi$-components  forms the following partition of $\bR^2$:
$$\{ \accentset{\circ}{\Gamma},  (-1,1)\times \{0\}, (-1,1)\times (-\infty,0) \} \cup \cF \cup \cG \cup \cH^+ \cup \cH^{-} , \text{ where}$$
\begin{itemize}
\item  
$\cF$ is the family of all singletons $\{(x,y)\}$ such that $x^2+y^2=1, y\geq 0$
\item
$\cG$ is the family of all half lines $\{t(x,y):t>1\}$ such that $x^2+y^2=1, y\geq 0$
\item
$\cH^+$ is the family of all half lines $\{(1,0)+t(x,y):t> 0\}$ such that $x^2+y^2=1, y< 0 \leq x$
\item
$\cH^{-}$ is the family of all half lines $\{(-1,0)+t(x,y):t> 0\}$ such that $x^2+y^2=1, y< 0 , x\leq 0$.
\end{itemize} 
\end{example}

 The following key proposition shows that these sets are intimately linked with the components of the domain of $\cM(\mu,\nu)$. 
\begin{proposition}\label{prop:as_aff_comp}
Fix $x\in \bB$ and $(\phi^n)_{n\in \bN} \subseteq \cC$. If $\an{\nu -\mu, \phi^n}\to 0$ then for any $\theta \in \cM(\mu,\nu) $ and disintegration $\theta=\mu\otimes \gamma$, $ \gamma(x, \cdot)$ is concentrated on $\overline{A_x((\phi^n)_n)}$  for $\mu$ a.e. $x$.
\end{proposition} Proposition \ref{prop:as_aff_comp} has the following interesting corollary. 
\begin{corollary}
\label{MeasEq}
If $\an{\nu -\mu, \phi^n}\to 0$ for $(\phi^n)_{n\in \bN} \subseteq  \cC $ then 
$\mu_{|\bB \setminus \overline{\cV((\phi^n)_n)}}=\nu_{|\bB \setminus \overline{\cV((\phi^n)_n)}}$.
\end{corollary} 
\begin{proof}
Recall that by Propositions \ref{prop:PropOfA_x} and \ref{prop:as_aff_comp}, the sets $A_x$ form a convex partition such that for $\theta=\mu\otimes \gamma\in \cM(\mu,\nu)$, $\gamma(x,\overline{A_x})=1$ $\mu(dx)$-a.e.
In particular, $A_x=\{x\}$ if $x\in \bB\setminus \cV$ and $A_x\subseteq \cV$ if $x\in \cV$ so that $ \gamma(x, \cdot)=\delta_x$ for $\mu$ a.e. $x\in \bB\setminus \cV$ and $\gamma(x,\cdot)$ is concentrated on $\bar{\cV}$ for $\mu$ a.e. $x\in \cV$. Consequently, if $B\subseteq \bB $ is Borel we get that 
\begin{align*}
\nu(B)=\int_{\bB} \mu(dx) \gamma(x,B)= \int_\cV \mu(dx) \gamma(x,B\cap \bar{\cV}) +\int_{\bB \setminus \cV} \mu(dx) \delta_x(B)
\end{align*} 
and  if moreover $B\subseteq \bB \setminus \bar{\cV}$ we get 
$ \nu(B)=\int_{\bB \setminus \cV} \mu(dx) 1_B(x)=\mu(B) .$
\end{proof} 

We now easily obtain the result about the convex order stated in the introduction.
\begin{proof}[Proof of Corollary \ref{MeasEqintro}]
 Let $\phi_n$ be the infimal convolution between $\phi$ and $n|| \cdot ||$, then $\phi_n \in \cC$ and $\phi_n \uparrow \phi$ (see \cite[Proposition 2.2.4 Chapter 1 and Proposition 3.1.4 Chapter 4]{HiLe93V1}). 
 Let $a$ be an affine function such that $a\leq \phi_1$ and $\lambda \in \cM$,  
 then $(\phi_n)^- \leq |a| \in L^1(\lambda)$ and $\phi_n^+ \uparrow \phi^+$ and so by dominated convergence and by monotone convergence we get that $\int \phi_n d\lambda \uparrow \int \phi d\lambda$. Since we assumed $\int \phi d \mu=\int \phi d \nu<\infty$ we get that so $\an{\nu -\mu, \phi^n}\to 0$. 
 Since $\phi$ is strictly convex and $\phi_n \to \phi$ we get that $ \cV((\phi^n)_n)=\emptyset$, thus Corollary \ref{MeasEq} gives the thesis.
 \end{proof}

We close this section with application of Proposition \ref{prop:as_aff_comp} to our motivating examples.
\begin{example}[Examples \ref{ex:discrete}--\ref{ex:mixed} continued.]
\label{ex:convex_components}
We continue the discussion of our motivating examples. Recall the pairs of measures in convex order: $\mu^k\preceq_c\nu^k$ and $\tilde \mu^k\preceq_c\nu^k$. 
We study the sets $A((\phi^n))_x$ for $\an{\nu -\mu, \phi^n}\to 0$. In fact, for this example, we restrict out attention to constant sequences $\phi^n=\phi$ with $\an{\nu -\mu, \phi}=0$. Consider $\phi((t,s))=f(t)$ for a strictly convex and \Lz  $f$ so that, in particular, $A(\phi)_{(t,s)}=\{t\}\times \bR$. 
Analogously, consider $\psi((t,s))=g(s)$, where $g$ is  Lipschitz, convex, equal to $0$ on $[-1,1]$ and strictly convex on $(\infty,-1)$ and on $(1,\infty)$. It follows that $A(\psi)_{(t,s)}=\bR\times (-1,1)$ for $s\in (-1,1)$ and $t\in \bR$. 
It is easy to compute the difference of integrals of $\phi$ or $\psi$ against our measures.  First,
$$\an{\nu^k-\mu^k, \phi}=\iint \mu^k(dx) \gamma^k(x,dy) (\phi(y)-\phi(x))=0,$$ 
where $\theta^k=\mu^k\otimes \gamma^k\in\cM(\mu^k,\nu^k)$ was exhibited in Examples \ref{ex:discrete}--\ref{ex:continuous}, and the above follows since $\gamma^k((t,s),\cdot )$ is concentrated on $\{t\} \times \{-1,1\}$ and $\phi$ is constant on $\{t\} \times \bR$.  Second, we have  
 $\an{\nu^k-\mu^k, \psi}=0$ simply since $\mu^k,\nu^k$ are supported on $[0,1]\times [-1,1]$. 
By Proposition \ref{prop:as_aff_comp}, for any $\theta\in \cM(\mu^k,\nu^k)$ with disintegration $\theta=\mu^k\otimes \gamma$ we have $\gamma((t,s),\cdot)$ is supported on 
$$\overline{A(\phi)_{(t,s)}\cap A(\psi)_{(t,s)}}= \{t\}\times [-1,1]$$ $\mu(d(t,s))$--a.e. Combined with our explicit construction of $\gamma^k$ this shows that Theorem \ref{thm:main_result} holds with $C_{(t,s)}=\{t\}\times (-1,1)$ for $s\in (-1,1)$ and $(t,0)$ in the support of $\mu^k$ and $C_{(t,s)}=\{(t,s)\}$ otherwise, $k\leq \infty$.

Similarly to above, $\an{\nu^k-\tilde\mu^k, \psi}=0$ and also $\an{\nu^k-\tilde\mu^k, \chi}=0$, where $\chi((t,s))=h(t)$ for a Lipschitz, convex function $h$ equal to $0$ on $[0,1]$ and strictly convex on $(\infty,0)$ and on $(1,\infty)$ so that $A(\chi)_{(t,s)}=(0,1)\times \bR$ for $t\in (0,1)$, while
 $A(\chi)_{(0,s)}=\{0\}\times \bR$ and $A(\chi)_{(1,s)}=\{1\}\times \bR$ since the point $x$ has to belong to the relative interior of $A(\phi)_x$, see \eqref{defA_x}. It follows that 
$C_{(t,s)}(\tilde\mu^k,\nu^k)\subset (0,1)\times (-1,1)$ for $(t,s)\in (0,1)\times (-1,1)$ and $C_{(0,s)}(\tilde\mu^k,\nu^k)\subset \{0\}\times (-1,1)$, $C_{(1,s)}(\tilde\mu^k,\nu^k)\subset \{1\}\times (-1,1)$. 
 From our example of $\tilde \theta^k\in \cM(\tilde \mu^k,\nu^k)$ we see that the inclusions may not be strict so that the convex components are indeed as asserted in Example \ref{ex:mixed}.
\end{example}

\subsection{Convex components describing support of martingale transports}\label{sec:convex_components}

We saw in Proposition \ref{prop:as_aff_comp} that for any martingale transport $\theta\in\cM(\mu,\nu)$ the mass from $x$ is diffused within $\overline{A_x((\phi^n)_n)}$ for $\mu$--a.e.\ $x$. This holds for any sequence $(\phi^n)_{n\in \bN} \subseteq \cC$ with $\an{\nu -\mu, \phi^n}\to 0$. In consequence, one may be inclined to ask if $\gamma(x,\cdot)$, where $\theta=\mu\otimes \gamma$, is concentrated on the intersection of $\overline{A_x((\phi^n)_n)}$ over all such sequences $(\phi^n)_{n\in \bN} \subseteq \cC$? This, in general, is false. Indeed, for a fixed $x$, we can typically find a sequence $(\phi^n)_{n\in \bN} \subseteq \cC$ such that $\overline{A_x((\phi^n)_n)}$ is too small. 
In other words, the union over all $(\phi^n)_{n\in \bN} \subseteq \cC$ such that  $\an{\nu -\mu, \phi^n}\to 0$ of the  $\mu$-null set $\cN_{\mu}((\phi_n)_n)$ on which it does not happen that 
 $\gamma(x,\cdot)$ is concentrated on $\overline{A_x((\phi^n)_n)}$ is not a $\mu$-null set. To understand this we come back to Examples \ref{ex:discrete}--\ref{ex:mixed}.

\begin{example}[Examples \ref{ex:discrete}--\ref{ex:continuous} continued.]
\label{ex:Ax_toosmall}
We continue the discussion of our motivating examples and their convex components as computed in Example \ref{ex:convex_components}. We argue that $C_z$ may not be defined as the intersection of $A((\phi^n))$ over all sequences with $\an{\nu^k -\mu^k, \phi^n}\to 0$. Fix $x_0\in (0,1)$. Now let $\phi^n((x,y))=(y-n(x-x_0))^+$ so that $\phi^n$ is affine on $\{x\}\times [-1,1]$ for $x\notin (x_0-\frac{1}{n},x_0+\frac{1}{n})$. It follows that $\an{\nu^k -\mu^k, \phi^n}\to 0$ for any $2\leq k\leq \infty$. However $\phi^n((x_0,y))=y^+$ from which we see that $A((\phi^n))_{(x_0,0)}=\{(x_0,0)\}$ and hence
$$\bigcap_{(\phi^n)_{n\in \bN} \subseteq \cC: \an{\nu^k -\mu^k, \phi^n}\to 0}\overline{A((\phi^n))}_{(x,0)}=\{(x,0)\}\subsetneq C_x(\mu^k,\nu^k),\quad 2\leq k\leq \infty.$$
\end{example}

To circumvent the above problem, we would need to restrict ourselves to a suitable countable family of sequences of functions in $\cC$. This can be achieved by considering a suitably defined essential intersection instead of the simple intersection above. To describe our construction we need some additional definitions. Let $CL(\bB)$ be the set of non-empty closed subsets of $\bB$. There is number of well understood topologies one may put on $CL(\bB)$. For our purposes it is most convenient to equip $CL(\bB)$ with the Wijsman topology \cite{Wijsman:66} which is the weak topology generated by mappings $d_{\cdot}(x):CL(\bB)\to \bR$, $x\in \bB$. This topology is weaker than the Vietoris topology and stronger than the Fell topology. To us, is has two main advantages. First, it makes $CL(\bB)$ into a Polish space as shown by Beer \cite{Beer:91}. Second, it generates the Effros $\sigma$--algebra which implies that weak measurability of closed--valued multifunctions can be treated similarly to regular functions, see \cite{Beer:91} and the references therein. Finally, let $\cc$ be the set of closed convex subsets of $\bB$. Then $\cc$ is a closed subset of $CL(\bB)$ and hence also Polish with its Wijsman topology \cite{Wijsman:66}. We equip it with partial ordering given by set inclusion. Then, in a recent work, Larsson \cite{Larsson:17} showed that one can build a strictly increasing measurable map from $\cc$ to $\bR$. With such a map, one can follow the usual arguments, to establish existence of essential infimum of a family of $\cc$--valued random variables, see \cite{Larsson:17} for details. This allows us to give a proper definition of convex component for $\cM(\mu,\nu)$ which avoids the problems highlighted in Example \ref{ex:Ax_toosmall} above.

For a fixed sequence $(\phi^n)_n\in \cC$, we see the mapping
\begin{equation}\label{eq:Axphin_mapping}
 \bB\ni x\to \overline{A_x ((\phi^n)_n)}\in \cc
\end{equation}
as a $\cc$--valued random variable on $(\bB,\cB(\bB),\mu)$. It follows from the structure of Wijsman topology, that the measurability of the above mapping is equivalent to its Borel measurability as a multifunction, see Hess \cite{Hess:86}. We believe this follows readily but leave the details aside\footnote{We plan to add these in the subsequent version of the paper.}.
We are interested in the collection of such variables over 
$$(\phi^n)_n\in I:=\{ (\phi^n)_n \in \cC \text{ such that } \an{\nu-\mu, \phi^n}\to 0 \}.$$
As explained above, we can take their essential infimum with respect to $\mu$ which exists, is unique $\mu$-a.s., measurable and $\cc$--valued. 
Further, it may be obtained as an infimum over a countable family: there exits sequence $(\phi^{k,n})_n\in I$, $k\geq 1$,  such that
\begin{equation}\label{eq:cvcmp_def}
\cnvmap_x(\mu,\nu):=\bigcap_{k\geq 1} \overline{A_x((\phi^{k,n})_n)} 
\end{equation}
satisfies
\begin{equation*}\label{eq:essinfdef}
\cnvmap_x(\mu,\nu)= \mu-\mathrm{essinf}_{(\phi^n)_n\in I} \overline{A_x((\phi^n)_n)} \qquad \mu(dx)\textrm{--a.e.}
\end{equation*}
We now want to define the convex component of $x$ as the largest relatively open convex subset of $\cnvmap_x$ which contains $x$. 
This can be achieved using faces of a convex set, which we also exploit in Section \ref{subsec:convex_face} to characterise the asymptotically convex components $A_x$. We recall here that given a convex set $K\subseteq \bB$ and $x\in K$, 
there exists a unique face\footnote{While we follow \cite{HiLe93V1, Rock:70}, we warn the reader that some other authors simply call `face' what \cite{HiLe93V1, Rock:70} call `exposed face', and that this distinction is important since not all faces are exposed (see immediately after  \cite[Chapter 3, Remark 2.4.4]{HiLe93V1}).} $F_x$ of $K$ that contains $x$ in its relative interior $\ri(F_x)$  (see  \cite[Theorem 18.2]{Rock:70}), which \cite[Theorem 18.1]{Rock:70} shows to be the smallest face of $K$ containing $x$ and, by \cite[Theorem 18.2]{Rock:70},  is also the maximal subset in $\cO$ included in $K$ and containing $x$. 
\begin{definition}\label{def:convex_component}
For $x\in \bB$ the set
$$C_x=\cvcmp:= \ri\left(F_x\left(\cnvmap_x(\mu,\nu)\right)\right)\in \cO$$ 
is called the \emph{convex component of $x$} and the set 
$$\cD=\cD(\mu,\nu):=\{x\in \bB: C_x\neq \{x\}\}$$ is called the \emph{domain}.\\ 
\end{definition}
\begin{remark}
We stress that convex components are defined $\mu(dx)$--a.e. The particular definition in \eqref{eq:cvcmp_def} could be modified on a $\mu$--null set as long as the resulting sets are also a convex partition of $\bB$. 
\end{remark}
\begin{proof}[Proof of Theorem \ref{thm:main_result}]
Recall that $x\in A_x((\phi^n)_n)$ and hence $x\in \cnvmap_x(\mu,\nu)$ and hence Definition \ref{def:convex_component} is well posed. We need to show that the convex components $C_x$, $x\in \bB$, of Definition \ref{def:convex_component} form a convex partition of $\bB$, $x\in C_x$, and that 
for any $\theta \in \cM(\mu,\nu) $ and disintegration $\theta=\mu\otimes \gamma$, $ \gamma(x, \overline{C_x})=1$ $\mu(dx)$-a.e. Note that, by Propositions \ref{prop:PropOfA_x} and \ref{prop:as_aff_comp}, these properties are true for $A_x((\phi^{k,n})_n)$, for each $k\geq 1$. In particular, since $\cnvmap_x=\lim_{K\to \infty} \bigcap_{k\leq K} \overline{A_x((\phi^{k,n})_n)}$, we see that
 $\gamma(x,\cnvmap_x)=1$, $\mu(dx)$-a.e. and, by Theorem \ref{ConconA_x}, $\gamma(x,\overline{C_x})=1$, $\mu(dx)$-a.e.\\
As recalled above, for a convex set $K$ and $x\in K$, $\ri(F_x(K))$ is the largest relatively open set which includes $x$ and is contained in $K$. It follows that for two convex sets $K_1\subseteq K_2$,  $x\in K_1$, we have $\ri(F_x(K_1))\subseteq \ri(F_x(K_2))$. In particular, for any $(\phi^n)_n\in I$
$$C_x=\ri\left(F_x\left(\overline{C}_x(\mu,\nu)\right)\right)\subseteq \ri\left(F_x\left(\overline{A_x((\phi^n)_n)}\right)\right)=A_x((\phi^n)_n),\quad \mu(dx)\textrm{-a.e.},$$
where the last equality follows from the characterisation in Lemma \ref{AffCompIsFace} below. Note that we may take the above to hold for $((\phi^{k,n})_n)$, $k\geq 1$, in \eqref{eq:cvcmp_def}, $\mu(dx)$-a.e. Now suppose $y\in C_x$. Then, by the above inclusion and by Proposition \ref{prop:PropOfA_x}, $A_x((\phi^{n,k})_n)=A_y((\phi^{n,k})_n)$ so, by \eqref{eq:cvcmp_def}, we have $\cnvmap_x=\cnvmap_y$. It now follows that $x,y$ are in the relative interior if the same convex face of this set and hence their convex components are equal. This concludes the proof.
\end{proof}
Finally, observe that Theorem \ref{thm:main_result}, combined with the general argument given for the proof of Corollary \ref{MeasEq}, readily imply the following result.
\begin{corollary}
\label{Mu=NuOutsideD}
$\mu_{|\bB \setminus \bar{\cD}}=\nu_{|\bB \setminus \bar{\cD}}$.
\end{corollary} 

\section{Proofs and further properties of affine components}\label{sec:cnv_comp_proofs}
We turn now to proofs of the results announced in Section \ref{sec:cnv_comp_results}.
We first establish Proposition \ref{prop:def_aff_comp} and then prove Proposition \ref{prop:as_aff_comp}  by characterising $A_x$ as the convex face containing $x$. For the latter proof, we establish certain results in convex analysis which are of independent interest, see Theorem \ref{ConconA_x} below. 
 We start however with a simple result, exploited in all of the proofs, which asserts that because of convexity, it is enough to consider very special affine functions to determine whether $(\phi^n)_{n\in \bN} \subseteq \cC$ is  asymptotically affine on a set $V\in \cO$: it is enough to consider affine functions $b_n$ supporting $\phi^n$ at a \emph{fixed}\footnote{I.e. a point $p$ not dependent on $n$; otherwise the result is false.} point $p\in V$. This fact relies crucially on the assumption that\footnote{As it is clear from the proof, Lemma \ref{AnySuppFn} would hold for arbitrary $V$ if we assumed that $p\in \ri((co(V))$.} $V\in \cO$ --- for a counterexample one may consider $V=[0,1), p=0$ and $\phi^n(t)=t^+$, where $\phi_{|V}$ is affine and the  function $b^n:=0$ supports $\phi^n$ at $p$ yet $(\phi^n-b^n)(t)=t$ is not identically $0$ on $V$.

\begin{lemma}
\label{AnySuppFn}
Given $(\phi^n)_{n\in \bN} \subseteq \cC$  and  $p\in V\in \cO$,  choose\footnote{Such an affine function $b_n$ always exists  (i.e. $\partial \phi^n(p)\neq \emptyset$) since $\phi$ is convex: see \cite[Chapter 6, Definition 1.1.4 and Theorem 1.2.2]{HiLe93V1}.}   $b_n$  affine s.t.   $b_n\leq \phi^n$ and $b_n(p)=\phi^n(p)$. 
Then $\phi^n - b_n \to 0$ on $V$ iff $(\phi^n)_n$ is asymptotically affine on $V$.
\end{lemma} 
\begin{proof}
By definition if $\phi^n - b_n \to 0$ on $V$ then $(\phi^n)_n$ is asymptotically affine on $V$. Conversely, let $a_n$ be affine and such that $\phi^n -a_n \to 0$ on $V$.
If $V$ is a singleton then trivially $\phi^n - b_n \to 0$ on $V$. If  $V$ is not a singleton by restricting ourselves to its affine hull we can assume w.l.o.g. that $V$ is open (in $\bR^N$ with $N\geq 1$).
We can then apply \cite[Chapter 6, Theorem 6.2.7]{HiLe93V1}  to $f_n=\phi^n -a_n$ and get that $\max \{ ||d||_{\bB}: d\in \partial f_n(p) \} \to 0$. Since $\phi^n -b_n\in \cC_+$ equals $0$ at $p$, it achieves it minimum $0$ at $p$, and so  $0 \in \partial (\phi^n-b_n)(p)$, i.e. $\nabla b_n(p)\in \partial \phi^n(p)$. This gives that $\nabla b_n(p) - \nabla a_n(p) \in \partial f_n(p)$, so we get that  $\nabla b_n(p) - \nabla a_n(p) \to 0$. Since $a_n,b_n$ are affine and $b_n(p)- a_n(p) = \phi^n(p)-a_n(p) \to 0$ we get that  $b_n-a_n\to 0$ on $\bB$, so $\phi^n - b_n=(\phi^n - a_n) +(a_n - b_n) \to 0+0=0$ on $V$.
\end{proof}

\subsection{Proof of Proposition \ref{prop:def_aff_comp}}
Proposition \ref{prop:def_aff_comp} follows from the series of lemmas below:
Lemmas \ref{AsympAffOnCo} and \ref{ConvComp} imply that $(\phi^n_{|A_x})_n$ is asymptotically affine and  $x\in A_x \in \cO$, and Lemma \ref{VinC} gives that if  $x\in V \in \cO$ and $(\phi^n_{|V})_n$ is asymptotically affine then $V\subseteq A_x$, which completes the proof. 
\begin{lemma}
\label{AsympAffOnCo}
If $(\phi^n)_{n\in \bN} \subseteq \cC$ is asymptotically affine on  $ V_i\in \cO$ for each $i\in I$ and $\cap_{i\in I} V_i\neq \emptyset$ then $(\phi^n)_n$ is asymptotically affine on $\co(\cup_{i\in I} V_i)$.
\end{lemma} 
\begin{proof}
By applying Lemma \ref{AnySuppFn} we can choose an affine $b_n$ supporting $\phi^n$ at $p\in \cap_{i\in I} V_i$ and get that $\phi^n - b_n\to 0$ on each $V_i$ and so on $\cup_{i\in I} V_i$. If $\{i_1,\ldots, i_n\}\subseteq I$, $(t_{i_j})_{j=1,\ldots, n}$ are non-negative and such that $\sum_{j} t_{i_j}=1$, and $x_{i_j}\in V_{i_j}$ we get that 
$$ \textstyle 
0\leq (\phi^n - b_n)(\sum_{j=1}^n t_{i_j} x_{i_j}) \leq \sum_{j=1}^n t_{i_j}  (\phi^n - b_n)(x_{i_j}) \to 0 .$$
The thesis follows since the set of all points of the form $\sum_{j=1}^n t_{i_j} x_{i_j}$  is  $\co(\cup_{i\in I} V_i)$: see \cite[Chapter 3, Example 1.3.5]{HiLe93V1}. 
\end{proof}

 \begin{lemma}
\label{ConvComp}
 If  $V_i\in \cO$  for each $i\in I$ and $x\in \cap_{i\in I} V_i $ then $\cap_{i\in I} \overline{V_i} =\overline{\cap_{i\in I} V_i }$  and  $x\in \ri(\co(\cup_{i\in I} V_i))\in \cO$.
\end{lemma} 
\begin{proof}
Trivially $\cap_{i\in I} V_i $ is included in the closed set $\cap_{i\in I} \overline{V_i} $, and so is its closure. For the opposite inclusion take $y\in \cap_{i\in I} \overline{V_i}$, then \cite[Lemma 2.1.6]{HiLe93V1} gives that $(y,x]\subseteq \cap_{i\in I} V_i $ and thus 
 $y \in \overline{\cap_{i\in I} V_i }$, proving that $\cap_{i\in I} \overline{V_i} =\overline{\cap_{i\in I} V_i }$.
 
 The set  $C:=\co(\cup_{i\in I} V_i)$ is  convex and $x\in C$, so the  relative interior $\ri(C)$ of $C$ is convex and relatively open. 
Now by restricting our attention to $\aff(C)=\aff(\cup_{i\in I} V_i)$ we assume w.l.o.g. that $C$ has full dimension $N$. 
Assume by contradiction that $x\in C\setminus \ri(C)$; then  there exists $v\in \bB\setminus \{0\}$ such that the closed half space $H_x^v:=\{ y: \an{y-x,v}\geq 0\}$ contains $C$ (see \cite[Lemma 4.2.1]{HiLe93V1}). Since $C=\co(\cup_{i\in I} V_i)$ has full dimension, it is not contained in the hyperplane $\partial H_x^v$, so there exists $\bar{i}\in I$  and  $y \in  V_{\bar{i}}$ such that $ y \notin \partial H_x^v$. Since $[y,x]\subseteq V_{\bar{i}}\in \cO$ there exists $z\in V_{\bar{i}} \cap (\overrightarrow{[y,x)} \setminus [y,x])$, which is absurd since 
$H_x^v \cap (\overrightarrow{[y,x)} \setminus [y,x])=\emptyset$ and $ V_{\bar{i}} \subseteq C_x \subseteq H_x^v$. Thus $x\in \ri(C)$.
\end{proof} 

\begin{lemma}
\label{VinC}
Assume that $C \subseteq \bB$ is convex and $V\in \cO, V\subseteq C$. If $V\cap \ri(C) \neq \emptyset$ then $V \subseteq \ri(C)$.
\end{lemma} 
\begin{proof}
Assume by contradiction that there exist  $y\in V\cap \ri(C), z\in  V \setminus \ri(C)$, then 
$C \cap (\overrightarrow{[y,z)} \setminus [y,z]) =\emptyset$ since otherwise $z$ would belong to $\ri(C)$ (see \cite[Lemma 2.1.6]{HiLe93V1}). However  $V\in \cO$ gives that $V \cap (\overrightarrow{[y,z)} \setminus [y,z]) \neq \emptyset$, contradicting $ V\subseteq C$.
\end{proof} 

\subsection{Characterisation of $A_x$ as a convex face \& proof of Proposition \ref{prop:as_aff_comp}}\label{subsec:convex_face}
We turn now to the proof of Proposition \ref{prop:as_aff_comp} which relies on a characterisation of $A_x$  in terms familiar to the convex analyst. We will repeatedly use without further notice the fact that, given a convex set $D\subseteq \bB$ and $x\in D$, 
there exists a unique face\footnote{While we follow \cite{HiLe93V1, Rock:70}, we warn the reader that some other authors simply call `face' what \cite{HiLe93V1, Rock:70} call `exposed face', and that this distinction is important since not all faces are exposed (see immediately after  \cite[Chapter 3, Remark 2.4.4]{HiLe93V1}).} $F_x$ of $D$ that contains $x$ in its relative interior $\ri(F_x)$  (see  \cite[Theorem 18.2]{Rock:70}), which \cite[Theorem 18.1]{Rock:70} shows to be the 
the smallest face of $D$ containing $x$ (which exists since trivially any intersection of faces of $D$ is a face of $D$) and which is given by 
\begin{align}
\label{Face_x}
 F_x=F_x(D)=\{y\in D : \exists z\in D , t\in (0,1)  \text{ such that }  x=ty+(1-t)z\} ,
\end{align} 
 see\footnote{The given reference  \cite{LuJaMaNeSp10} assumes that $D$ is compact at the beginning of Section 2.3 (which contains Corollary 2.67), but clearly this is not needed to show our \eqref{Face_x}. Moreover  \cite[Corollary 2.67]{LuJaMaNeSp10} contains a small typo (clearly $t=\lambda$ should belong to $(0,1)$, not $[0,1)$).} \cite[Corollary 2.67]{LuJaMaNeSp10}.
\begin{lemma}
\label{AffCompIsFace}
Given $(\phi^n)_{n\in \bN} \subseteq  \cC $ let $b^n$ be affine and s.t.   $b^n\leq \phi^n$ and $b^n(x)=\phi^n(x)$.
Then the face $F_x=F_x((\phi^n)_n)$ of the convex set $\{y:(\phi^n-b^n)(y)\to 0 \}$ for which $x\in \ri(F_x)$  satisfies $A_x=\ri(F_x)$ and  $\overline{A_x}=\overline{F_x}$.
\end{lemma} 
\begin{proof}
Since $b^n \leq \phi^n$, the set $D:=\{y:(\phi^n-b^n)(y)\to 0 \}$ is convex.
According to \cite[Theorem 18.2]{Rock:70}  the relative interiors of its non-empty faces constitute the maximal subsets of $D$ in $\cO$.
Thus  
  $\ri(F_x)= A_x$ by  Lemma \ref{AnySuppFn} and by maximality of $F_x$ and of $A_x$.
Since any convex set $C$ satisfies $\overline{C}=\overline{\ri(C)}$ (see \cite[Chapter 3, Proposition 2.1.8]{HiLe93V1}) we get 
$\overline{F_x}=\overline{\ri(F_x)}= \overline{A_x}$.
\end{proof} 

An interesting by-product  of Lemma \ref{AffCompIsFace} is the fact that the relative interior of the face of $\{ \phi^n -b^n \to 0\}$ which contains $x$ does not change if we choose a different supporting function $b^n$ (even if this changes $\{ \phi^n -b^n \to 0\}$). To better understand this, consider the following  simple example in dimension one and with constant sequences  $\phi^n=\phi$ and $b^n=b$. Let $x=0$ and $\phi(t)=t^+$, then  $b_1:=0$ and $b_2(t)=t/2$ both support $\phi$ at $x$ and $\{ \phi=b_1\}=(-\infty,0]$ is very different from $\{ \phi=b_2\}=\{0\}$, yet $A_{0}(\phi)=\{0\}$ is  the face $F_x$ of $D$ which contains $x=0$ in its relative interior both for $D=(-\infty,0]$ and for $D=\{0\}$.

 \begin{theorem}
\label{ConconA_x}
Let $\alpha \in \MN_1$ be concentrated on a convex set  $D\subseteq \bB$.
Let $b(\alpha):=\int y \alpha(dy)$ be the barycenter of $\alpha$ and  $F_{b(\alpha)}$   the face of $D$  that contains $b(\alpha)$ in its relative interior. Then 
 $b(\alpha)\in D$ and  $\alpha$ is concentrated on $F_{b(\alpha)}$.
\end{theorem} 

If $D$ is a compact subset of a locally convex space,  the previous theorem is classical and holds in great generality (for example  see 
 \cite[Theorem 2.29 and Corollary 2.71]{LuJaMaNeSp10}).
 If $D$ is not compact even the conclusion $x\in D$ is generally false in infinite dimension (see for example \cite[Main Theorem]{We76}, or see \cite[Section 2.3.B]{LuJaMaNeSp10} for a more detailed study). Notice that the convex set $\{\phi^n-b^n\to 0 \}$ in Lemma \ref{AffCompIsFace} is  not necessarily even closed, so we are required to study general convex sets.
  
To build intuition for Theorem \ref{ConconA_x} consider a simple the two dimensional case: $D =[0,1]\times [-1,1]$ and take $x=(0,0)$, so that $F_x= \{0\}\times (-1,1)$. If $\alpha$ has barycenter $(0,0)$ and $\alpha((0,\infty)\times \bR)>0$ then, to have barycenter $(0,0)$, $\alpha$ must satisfy $\alpha((-\infty,0)\times \bR)>0$. So, if $\alpha$ is  concentrated on $D$ then it is actually concentrated on the smaller set  $F_x$. The general situation however is more involved since not all faces are exposed and we do not assume that $D$ is closed.

\begin{proof}[Proof of Theorem \ref{ConconA_x}] The fact that  $b(\alpha)\in D$ is proved in \cite[Lemma 1, Section 2.8, page 76]{Fe14Ma}. 
By definition of  face $C:=D\setminus F_x$ is convex.  Now assume by contradiction that  $\alpha(C)>0$.  Clearly $\alpha$ is not concentrated on $C$ since otherwise $x:=b(\alpha)\in C$, contradicting $x\in F_x$. Thus $\lambda:=\alpha(C)\in (0,1)$, so we can define mutually singular probabilities $\beta:=\lambda^{-1}\alpha_{|C} $ and $\gamma=(1-\lambda)^{-1}\alpha_{|F_x}$ such that  $\alpha=\lambda\beta+(1-\lambda)\gamma$ and as stated above   the barycenters of $\beta$ and $\gamma$ belong to the convex sets where they  are concentrated, and so $b(\beta)\in C, b(\gamma)\in F_x$.
Since $F_x \ni b(\alpha)=\lambda b(\beta)+(1-\lambda)b(\gamma)$ by definition of face we get that $b(\beta)\in F_x$, contradicting $b(\beta)\in C$. 
\end{proof}

\begin{proof}[Proof of Proposition \ref{prop:as_aff_comp}]
Since $\partial \phi^n$ is an upper hemicontinuous 
 multifunction (see \cite[Chapter 6, Theorem 6.2.4]{HiLe93V1}), there exists a Borel measurable selector $\dot{\phi^{n}}$ of $\partial \phi^n$, i.e. a Borel  function  such that $\dot{\phi^{n}}(x) \in\partial \phi^n(x)$ for all $x\in \bB$:  see for example  \cite[Theorem 18.13 and Lemmas 17.4 and 18.2]{AlBo99}.
Fixed such a Borel selector we define 
\begin{align}
\label{Delta}
\Delta_x \phi^n(y):=\phi^n(y)-(\phi^n(x) + \an{\dot{\phi^{n}}(x),y-x})  .
\end{align} 
Notice that independently from the choice of the kernel $\gamma$ and of the selector $\dot{\phi^{n}}$, one has $\Delta_x \phi^n(y) \geq 0$ and 
\begin{align}
\label{IntDelta}
\textstyle
0 \leftarrow \an{\nu-\mu, \phi^n }= \int  \mu(dx) \int  \gamma(x,dy) \Delta_x \phi^n(y) =\int \theta(d(x,y)) \Delta_x \phi^n(y) 
\end{align} 
and so $\Delta \phi^n\to 0$ in $L^1(\theta)$.
 Passing to a subsequence (without relabeling)  we get that 
$\theta$ a.e. $\Delta \phi^n\to 0$, i.e.  for $\mu$ a.e. $x$ we have that  $\gamma(x,\cdot)$ is concentrated on $\{ y:\Delta_x \phi^n(y)\to 0 \}$. Fix any one such $x$ and apply Theorem \ref{ConconA_x} with $\alpha=\gamma(x,\cdot) $ (so that $b(\alpha)=x$) and $D=\{ y:\Delta_x \phi^n(y)\to 0 \}$  to obtain that $\alpha$ is concentrated on $F_x$ and thus a fortiori on  $\overline{F_x}$, which by  Lemma \ref{AffCompIsFace}  equals 
$\overline{A_x((\phi^n)_n)}$. 
\end{proof} 

\bibliography{mot}
\bibliographystyle{abbrv}

\end{document}